%% file: main.tex
\documentclass[11pt]{article}
\usepackage{amscd,amssymb,stmaryrd}
\usepackage{amsmath}
\usepackage{graphicx}
\usepackage{subcaption}
\usepackage[utf8]{inputenc}
\usepackage[export]{adjustbox}
\usepackage{wrapfig}
\usepackage{amsthm}
\usepackage{color}
\usepackage{mathtools}
\usepackage{hyperref}
\usepackage[english]{babel}
\usepackage{booktabs}
\usepackage{float}
\usepackage[linesnumbered,ruled,vlined, ruled,vlined]{algorithm2e}
\usepackage{tikz}
\usepackage{pifont}% 
\usepackage{diagbox}
\newtheorem{theorem}{Theorem}[section]
\newtheorem{assumption}{Assumption}[section]

\newtheorem{lemma}[theorem]{Lemma}

\theoremstyle{definition}

\usetikzlibrary{decorations.pathreplacing,calligraphy}
\usepackage{pgfplots}
\pgfplotsset{compat=1.11}
\usepgfplotslibrary{groupplots}

\usetikzlibrary{shapes.misc}

\tikzset{cross/.style={cross out, draw=black, minimum size=2*(#1-\pgflinewidth), inner sep=0pt, outer sep=0pt},
cross/.default={1pt}}

\title{D2NO: Efficient Handling of Heterogeneous Input Function Spaces with Distributed Deep Neural Operators}
\date{}
\author{
Zecheng Zhang\thanks{Department of Mathematics, Florida State University, Tallahassee, FL 32304.}\and
Christian Moya\thanks{Department of Mathematics, Purdue University, West Lafayette, IN 47907.}\and
Lu Lu\thanks{Department of Statistics and Data Science, Yale University, New Haven, CT 06511.}\and
Guang Lin\footnotemark[2]\and
Hayden Schaeffer\thanks{Department of Mathematics, UCLA, Los Angeles, CA 90095.} }
\begin{document}
\maketitle
% ==============================
% Abstract
% ==============================

\begin{abstract}
Neural operators have been applied in various scientific fields, such as solving parametric partial differential equations, dynamical systems with control, and inverse problems. However, challenges arise when dealing with input functions that exhibit heterogeneous properties, requiring multiple sensors to handle functions with minimal regularity. To address this issue, discretization-invariant neural operators have been used, allowing the sampling of diverse input functions with different sensor locations. However, existing frameworks still require an equal number of sensors for all functions. In our study, we propose a novel distributed approach to further relax the discretization requirements and solve the heterogeneous dataset challenges. Our method involves partitioning the input function space and processing individual input functions using independent and separate neural networks. A centralized neural network is used to handle shared information across all output functions. This distributed methodology reduces the number of gradient descent back-propagation steps, improving efficiency while maintaining accuracy. We demonstrate that the corresponding neural network is a universal approximator of continuous nonlinear operators and present four numerical examples to validate its performance.
\end{abstract}

% ==============================
% Section: Introduction
% ==============================

\section{Introduction}
Operator Learning, introduced by \cite{chen1995universal}, involves learning operators that map one function to another. It extends classical function learning, which focuses on the mapping between vector spaces, and functional learning \cite{chen1993approximations}, which deals with functions mapping to fields. Operator Learning has become a pivotal tool in scientific machine learning \cite{karniadakis2021physics}. Recent works on this topic include various neural operators \cite{lu2021learning, li2020fourier, qian2022reduced, jin2022mionet, zhang2022belnet} and extensive studies on the approximation and convergence aspects of these neural operators \cite{chen1995universal, lanthaler2023curse, zhang2023discretization, deng2022approximation}. Additionally, applications of neural operators span various scientific and engineering domains. For example, researchers have developed algorithms that use neural operators to address parametric Partial Differential Equation (PDE) problems \cite{lu2021learning, zhu2023reliable, wang2021learning, mao2023ppdonet, mao2021deepm, cai2021deepm}, dynamical systems with control~\cite{lin2023learning,bhan2023operator, chen1995universal}, power engineering applications \cite{moya2023deeponet, sun2023deepgraphonet, ye2023deeponet}, design optimization problems~\cite{lu2022multifidelity, sahin2024deep}, and challenges related to multifidelity and multiscale downscaling \cite{howard2022multifidelity, lu2022multifidelity, zhang2023bayesian}. Notably, researchers have recently applied neural operators to climate predictions \cite{pathak2022fourcastnet} and uncertainty quantification \cite{lin2023b,psaros2023uncertainty,yang2022scalable}.

One of the popular approaches in this direction is the Deep Operator Neural Network (DeepONet) \cite{lu2021learning, lu2022multifidelity, lu2022comprehensive, zhu2023reliable, jin2022mionet,di2023neural,jiang2023fourier}. DeepONet utilizes two subnetwork architectures consisting of a branch network and a trunk network to approximate operators by constructing a basis for the output function space. The trunk networks are responsible for learning the basis, while the branch network acquires the basis coefficients. This results in an approximation achieved through a linear combination of the learned basis and coefficients. Another approach for neural operators is the Fourier neural operator (FNO) \cite{li2020fourier, zhu2023fourier, pathak2022fourcastnet}. FNO applies the Fourier neural transformation to convolve the input function with a learned translation invariant kernel function. It then utilizes skip connections and other mechanisms to nonlinearly generalize the approximation. A comprehensive comparison of DeepONet and FNO can be found in \cite{lu2022comprehensive}. Additionally, there are several other noteworthy neural operators documented, as cited here \cite{zhang2022belnet, qian2022reduced, zhang2023discretization, cao2022residual, liu2023prose}.

A significant advantage of DeepONet  \cite{lu2021learning, chen1995universal, jin2022mionet} and related structures \cite{zhang2022belnet, zhang2023discretization} is the ability to discretize output functions freely \cite{lin2021operator, lu2022comprehensive, zhang2022belnet}. The neural operator generates function values at specific domain points, which are determined by the trunk network input, resulting in enhanced adaptability. This flexibility allows the network to predict output function values at any domain point, accommodating different mesh definitions for output functions. However, a major challenge is achieving discretization invariance \cite{lu2022comprehensive, zhang2022belnet}. Specifically, the neural operator must remain invariant to the different discretizations applied to input functions. Several algorithmic extensions to DeepONet have been proposed, allowing different input functions to be discretized differently \cite{lu2022comprehensive, lin2021accelerated}. Additionally, the Basis Enhanced Learning (BelNet) extension of DeepONet, proposed by \cite{zhang2022belnet, zhang2023discretization}, has mathematically demonstrated discretization invariance. It is worth noting that these extensions still require the number of sensors used to sample input function values to be consistent across distinct input functions. In this paper, we introduce a novel and efficient distributed network architecture and training approach for neural operators. This approach permits individualized discretization of input functions using dedicated sensor distribution strategies.

Discretization invariance offers several advantages, particularly when working with datasets that have diverse properties \cite{zhu2023fourier}. For example, a dataset may consist of input functions with different levels of regularity. When using a uniform set of sensors to sample and discretize these input functions, it becomes necessary to accommodate the functions with the lowest regularity. This leads to a significant increase in the number of required sensors. However, it also results in higher costs for functions that are inherently smoother. To accommodate input functions with varying properties and regularity, we will draw inspiration from federated learning. Introduced in \cite{mcmahan2017communication}, federated learning is a machine learning framework that enables multiple clients to collaboratively train a consensus neural network model in a distributed manner, while keeping their training data local. This approach protects data privacy and reduces communication costs by coordinating local model updates with a central server, resulting in a significant reduction in data transfer volume.

A standard formulation of federated learning is a distributed optimization framework that addresses communication costs, client robustness, and data heterogeneity across different clients \cite{lsts20}. Communication efficiency is central to this formulation and motivates the most well-known efficient communication algorithm in federated learning: the federated averaging (FedAvg) algorithm~\cite{mcmahan2017communication}. FedAvg has been studied under realistic scenarios in~\cite{mcmahan2016federated,dean2012large}. Furthermore, several works have provided convergence proofs of the algorithm within the field of optimization~\cite{li2019convergence,li2021fedbn, khaled2019first,karimireddy2020scaffold,deng2021convergence}.

Traditionally, neural operators have been trained using a centralized strategy that involves transferring the training data to a central location. However, this approach hinders our ability to leverage high-performance distributed/parallel computing platforms. To address these limitations, the authors in~\cite{moya2022fed} proposed Fed-DeepONet. Fed-DeepONet enables distributed training of deep operator networks and has been empirically shown to achieve accuracy comparable to the centralized vanilla DeepONet while handling heterogeneous input functions. However, it should be noted that Fed-DeepONet requires the number of sensors for each local DeepONet to be the same, as it trains a consensus global DeepONet. In this paper, we provide a theoretical demonstration of Fed-DeepONet's ability to approximate nonlinear operators. Additionally, we design a dedicated distributed training strategy that does not require the same number of sensors, allowing the approach to handle heterogeneous input function spaces.

We summarize the contributions of this paper as follows:
\begin{enumerate}
    \item We propose the distributed deep neural operator (D2NO) method that allows for sampling different input functions using distinct sets of sensors. This training method can reduce the total number of back-propagation steps and the number of trainable parameters by utilizing dedicated neural networks to handle the diverse input function spaces.
    \item D2NO is motivated by the mathematical theory that underlies the algorithm. Since all functions share the same output basis, we construct dedicated networks to handle the diverse input functions. We then proceed to demonstrate the universal approximation of the network and algorithms we propose.
    \item The proposed methodology can be implemented using different neural operators. In this study, we apply it to the novel Deep Operator Neural Network (DeepONet). By utilizing the Distributed-DeepONet (D-DeepONet), we tackle complex heterogeneous datasets that contain input functions with different regularities and geometries. Our approach consistently improves predictive accuracy while maintaining efficiency.
\end{enumerate}
The rest of the paper is organized as follows. In Section \ref{sec_overview}, we will review the DeepONet and the universal approximation theory. Next, we present the methodology in Section \ref{sec_method}. Then, we prove the universal approximation theorem for Distributed-DeepONet in Section \ref{sec_theory}. Finally, we present several numerical experiments containing heterogeneous input function spaces in Section \ref{sec_numerical}.
% ==============================
% Section: Overview
% ==============================
\section{Overview}
\label{sec_overview}
\subsection{Deep Operator Neural Network (DeepONet)}
We start by defining an operator $G: V\rightarrow U$, where $V$ and $U$ are two function spaces. Next, we examine the Deep Operator Neural Network (DeepONet) \cite{lu2021learning, lu2022comprehensive, lu2022multifidelity, jin2022mionet, zhu2023reliable}. The DeepONet structure contains a branch and trunk network. The branch net processes a discretized input function, while the trunk network processes an arbitrary location within a given output function space. The DeepONet structure avoids needing to discretize the output function. This allows more flexibility in both the training and prediction processes, distinguishing it from neural operators that rely on output function discretization on fixed grid mesh.
\begin{figure}[H]
\centering
\scalebox{.78}{\input{deepo.tex}}
\caption{Stacked version DeepONet. $\bigotimes$ denotes the inner product in $\mathbb{R}^K$.}
\label{fig_don_structure}
\end{figure}
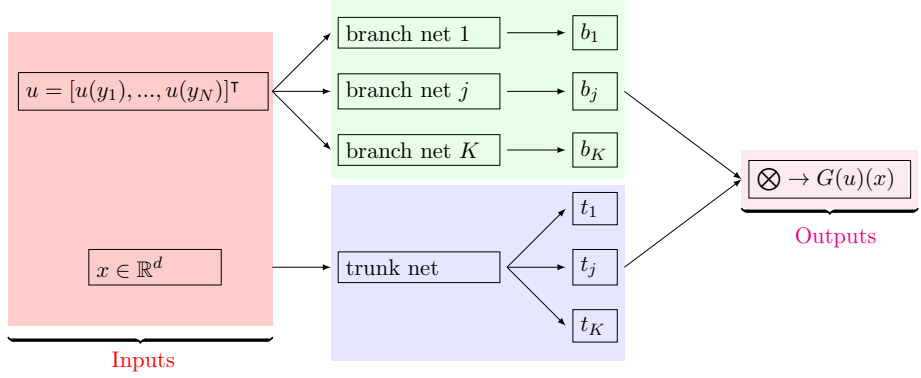

The DeepONet is theoretically based on the universal approximation theorem of nonlinear operators, which was established in \cite{chen1995universal, chen1993approximations, lu2021learning}, and extended in \cite{lu2021learning,zhang2023discretization}. The main results are stated below.
\begin{theorem}[Theorem 5, \cite{chen1995universal}]
Suppose $G: V\rightarrow U$ is continuous. Specifically, let $V$ be compact in the continuous function space $C(K_1)$, where $K_1$ is compact in a Banach space. Additionally, suppose $G(V)\subset C(K_2)$, where $K_2\subset\mathbb{R}^d$ is also compact. Then, for any $\varepsilon$, there exist sensors $\{x_i\}_{i = 1}^N\subset K_2$, networks $p_k: \mathbb{R}^N\rightarrow \mathbb{R}$, and $b_k: K_2\rightarrow \mathbb{R}$ such that,
\begin{align*}
    |G(u)(x) - \sum_{k = 1}^K p_k(\hat{u})b_k(x)| <\varepsilon,
\end{align*}
for all $x\in K_2$ and $u\in V$. Here $\hat{u} = [u(x_1), ..., u(x_N)]^\top$.
    \label{theorem_chen}
\end{theorem}

\subsection{Discretization-Invariant Extension of DeepONet}
Theorem \ref{theorem_chen} presents an existence argument stating that the sensors are the same across all input functions. This implies that all distinct input functions must be discretized on the same mesh. However, this limitation hinders the use of DeepONet, especially in cases where the dataset exhibits heterogeneous properties (for a demonstration, refer to the numerical experiments section \ref{sec_numerical}). A neural operator is considered discretization-invariant if it remains unaffected by variations in input function discretizations. In \cite{zhang2022belnet, zhang2023discretization}, an extension of DeepONet was shown to be discretization-invariant, and specifically, the authors proved the following theorem:
\begin{theorem}[Discretization-Invariance \cite{zhang2022belnet, zhang2023discretization}]
Suppose that $a \in TW$, $Y$ is a Banach space, $K_1\subset Y$, and $K_2\subset \mathbb{R}$ are all compact. Let $V\subset C(K_1)$ be a compact set and $G: V \rightarrow C(K_2)$ be a continuous and nonlinear operator. For any $\epsilon>0$, there exist integers $N,C,K,I$, weights and biases $W_x^k\in\mathbb{R}^{d}$, $b_x^k\in\mathbb{R}$, $W^k\in\mathbb{R}^{I\times C}$, $b^k\in\mathbb{R}^{I}$, $c^k\in\mathbb{R}^{I}$, subset of sensors $K_y\subset K^N_1$ and a trainable network $\mathcal{N}: K_y\rightarrow \mathbb{R}^{C\times N}$ with $K_y\subset K_2$.
Then the following inequality holds
\begin{align*}
    \left|
G\left(u\right)(x) - \sum_{k = 1}^K a(W_x^k\cdot x+ b^k_x) (c^k)^\top a\big(W^k\mathcal{N}(\hat{y})u(\hat{y}) + b^k \big)
    \right|<\epsilon,
\end{align*}
for all $x\in K_2$, $\hat{y} = [y_1, y_2, ..., y_N]^\top \in K_y$, and $u\in V$.
\end{theorem}
The basis enhanced learning (BelNet) extension partially relaxes the restriction on the input function discretization. However, the number of sensors (input function mesh size) should still be the same for all input functions. Additionally, $K_y\subset K$ which supports the discretization-invariant property must be carefully crafted (please check Remarks 3 and 4 in \cite{zhang2023discretization}). This limits the sensor placement and input function discretizations. 

When dealing with datasets that have input functions with different properties, it is important to place sensors that can adapt to functions with the lowest regularity within the dataset. Figure \ref{fig_different_geom} illustrates various functions with different regularities. To effectively capture functions with features like black curves with sharp peaks, a large number of sensors must be used. Similarly, for the magenta curves with smooth oscillations, sensors need to be strategically placed across the entire domain. Using a uniform sensor configuration for all input functions would require a significant increase in the number of sensors. This would result in higher computational costs and additional expenses, especially for smoother functions. We propose the D2NO which can effectively address this problem.

\subsection{Distributed/Federated Training of DeepONets}
In~\cite{moya2022fed}, the authors proposed Fed-DeepONet, a distributed/federated strategy that trains a consensus/universal DeepONet model for $C$ clients and a centralized server. This federated training strategy involves looping over the following three steps:
\begin{enumerate}
    \item \textit{Broadcast to clients:} The centralized server broadcasts the most up-to-date DeepONet consensus model to all clients.
    \item \textit{Client local updates:} For any $c \in \{1,2,\ldots,C\}$, the $c$-th client receives the most up-to-date DeepONet consensus model and performs $K \ge 1$ local stochastic gradient (or variants, e.g., Adam) DeepONet updates using the client's dedicated dataset.
    \item \textit{Global synchronization:} The centralized server aggregates the local DeepONet parameters into a unique consensus/universal DeepONet every $K$ local updates.
\end{enumerate}
Since Fed-DeepONet learns a consensus model, each client must use the same number of sensors. In the next section, we will expand on the previously mentioned federated training strategy and develop a dedicated distributed training strategy for DeepONets with $C$ clients. This strategy will efficiently handle heterogeneous input function spaces by allowing different numbers of sensors for each client $c \in \{1,2,\ldots,C\}$.
% ==============================
% Section: Methodology
% ==============================
\section{Methodology}
\label{sec_method}
Consider the nonlinear operator $G$ that maps an input function $u \in V$ to an output function $G(u)(x)$ evaluated at an arbitrary location $x \in K_2$. According to Theorem \ref{theorem_chen}, we can approximate $G(u)(x)$ as $\sum_{k = 1}^K p_k(u)b_k(x)$, where $b_k(x)$ and $p_k(u)$ are learnable functions. In the DeepONet literature, $p_k(\cdot)$ is commonly referred to as the branch net, while $b_k(\cdot)$ is known as the trunk network. To simplify our notation, we will omit the parameters of the learnable functions whenever possible. It is important to note that $b_k(x)$ can be seen as the basis of the output function space, which is a shared element among all $u$ and is independent of $u \in V$. On the other hand, $p_k(u)$ represents the coefficients in the linear combination, which depend on the specified input function $u$. Consequently, the approach taken is to develop a dedicated distributed approach that shares $b_k(x)$ across all samples to create a universal/consensus trunk network model. In this dedicated approach, individual clients independently manage their respective input function datasets using their branch nets (or projection in BelNet). They send their individual basis $b_k(x)$ to a centralized server to learn a universal/consensus model applicable to all samples. Please refer to Figure \ref{fig_federated_demo} for an illustration and a comparison with classical training (Figure \ref{fig_classical_demo}).
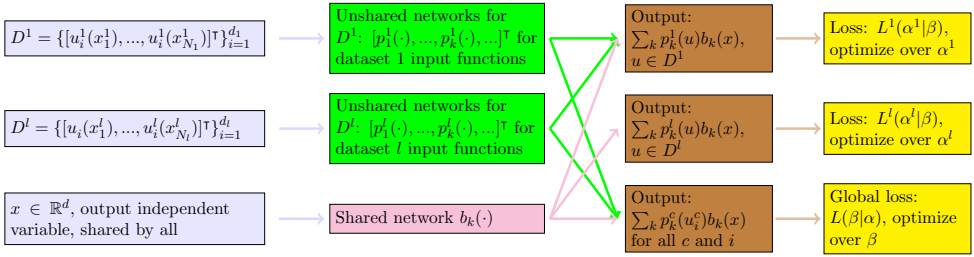
\begin{figure}[ht]
\centering
\scalebox{.6}{\input{federated_demon.tex}}
\caption{D2NOprocess demonstration. We process the sub-dataset $D^l$ by its dedicated branch nets (green box), together with the shared basis $b_k$  (pink box), we can construct the local loss $L^l$, and use the local loss to optimize the specialized branch net parameters. The global loss construction relies on the entire dataset as it is used to update the trunk net parameters (pink box) which are shared across different clients.}
\label{fig_federated_demo}
\end{figure}

\begin{figure}[ht]
\centering
\scalebox{.6}{\input{centralized_demon.tex}}
\caption{Classical DeepONet training process demonstration.}
\label{fig_classical_demo}
\end{figure}
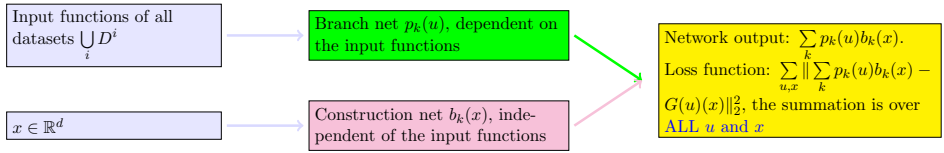

This means that each client independently manages input functions in its dataset through its unique branch (or projection) nets. Additionally, they transmit their local trunk networks to the central server responsible for acquiring and sending the universal/consensus basis $b_k(x)$ across all samples. Since each dataset is managed by its individual branch (or projection), it becomes possible to use designated sensors tailored to the input function within each dataset.

\textit{Distributed Deep Neural Operator (D2NO).} 
Let $D$ denote the entire heterogeneous dataset,  and $D^1, ...., D^C$ to be sub-datasets. Moreover, each $D^i$ exclusively comprises functions belonging to a single class, sharing similar regularity assumptions. Let us denote $\alpha_c$ as the trainable parameters of $c-$th client's dedicated neural network.
Let $\beta$ denote the trainable parameter of the shared network. Thus in this task, $\alpha_c$ are the branch net parameters, and $\beta$ are the trunk net parameters. 
We then denote $\alpha = \bigcup\limits_{c = 1}^C\alpha^c$ and $\theta = \alpha\bigcup \beta$.
The local loss function for $c-$th client based on its corresponding dataset $D^c$ will be:
\begin{align}
        L^c(\alpha^c\big| \beta) = \sum_{i = 1}^{N_c} \| G(u^c_i) - G_{\alpha^c, \beta}(\hat{u}^c_{i}) \|^2,
\end{align}
where $N_c$ denotes the number of input functions in dataset $D^c$, $u^c_i$ and $\hat{u}_i^c$ are the functions and its discretization. 
Please note that we normalize the loss function with the number of samples during the implementation, and present the unnormalized version for a better illustration of the concept.
We can then define the global loss function,
\begin{align*}
    L(\beta \big|\alpha) = \sum_{c = 1}^C L^c(\alpha^c|\beta),
\end{align*}
where $C$ denotes the total number of clients and datasets. We now summarize the pseudo-algorithm in Algorithm \ref{algo_sharing_trunk} with an illustration in Figure \ref{fig_federated_demo}.

\begin{algorithm}[h]
\caption{Distributed Deep Neural Network (D2NO) Algorithm}
\label{algo_sharing_trunk}
Initialization: weights $w_c$, all networks trainable parameters $\theta$, learning rate $\eta_c$ and $\eta$, the total number of iterations $N$\;
    \For{$k = 1$ to $N$}
    {
    Update the unshared parameters $\alpha^c_k$ for all clients, $k$ denotes the $k-$th optimization step\;
% \If{$k$ is a multiple of $K$}
% {
\For{$c = 1$ to $C$}
{
$$
\alpha_k^c = \alpha_k^c - \eta_c\nabla_{\alpha^c} L^c(\alpha^c|\beta),
$$
}% the end of the loop for clients
% }% the end of if
% \Else
% {$$\alpha_k^c = \sum_{c = 1}^C w_c\alpha^c_k.$$
% } % the end of Else
Update the shared construction net parameter $\beta$
$$
\beta_k = \beta_k  - \eta\nabla_{\beta} L(\beta|\alpha),
$$
}
\end{algorithm}
% =========================
% Efficiency of D2NO
% =========================
\subsection{Efficiency of Distributed Deep Neural Operator (D2NO)} 
This section provides a brief analysis of the efficiency of the proposed D2NO in terms of data and computational cost. To train scientific machine learning models, a significant amount of data is needed, much of which could be privacy-sensitive. Surrogates trained using this data have the potential to greatly enhance scientific and engineering discovery and simulation. However, transferring this data to a centralized location is costly, and the quantity and security of the data may make training in such a location impractical using the traditional DeepONet approach. The proposed D2NO addresses this issue by allowing the training data to remain distributed locally. Instead of sharing large datasets, the client only shares Deep Neural Operator parameters (or a partial model, e.g., the trunk network), which are aggregated using locally computed updates and which reduce training costs. 

To compare the computational training costs of the traditional DeepONet and D2NO, we proceed as follows. First, consider a centralized model with a vector of trainable parameters $\theta \in \mathbb{R}^p$.  Without loss of generality, let us assume we select one point $x^{(i)}$ in each output function's domain, we then minimize the following loss function,
$$
\mathcal{L}(\theta)  =  \frac{1}{N}\sum_{i=1}^N \mathcal{L}^{(i)} =  \frac{1}{N}\sum_{i=1}^N \left|G_\theta(\hat{u}_{i})(x^{(i)}) - G(u_{i})(x^{(i)})  \right|^2
$$

over the dataset of $N$ triplets $\mathcal{D}_{\text{cent}} = \{(u_m^{(i)}, x^{(i)}),G(u^{(i)})(x^{(i)})\}_{i=1}^N$, a standard deep learning package can compute the gradient of the loss function $\mathcal{L}$ with respect to each trainable parameter in $\beta$. In this analysis, we will focus on the number of gradients computed for each data loss function $\mathcal{L}^{(i)}$, where $i=1,\ldots,N$, with respect to the trainable parameters. This results in a total of $N \cdot p \cdot n_{\text{epochs}}$ gradient computations over $n_\text{epochs}$ of training. This analysis does not take into account batching, but extending it to the stochastic setting is direct.

For the D2NO case, we assume that the data is distributed among $C$ clients, so that $N=\sum_{j=1}^{C}n_j$, where $n_j$ represents the number of data points available to the $j$-th client for training a local model with a vector of $p' \le p$ trainable parameters. Therefore, for the $j$-th client, the number of gradients computed for each data loss function over $n_\text{epochs}$ is $n_j\cdot p' \cdot n_\text{epochs}$. As a result, we can establish the following inequality:
\begin{align*} \text{Cost of D2NO} &= \sum_{j=1}^{C} n_j \cdot p' \cdot  n_\text{epochs} \\ &= N \cdot p' \cdot n_\text{epochs} \\ &\le N \cdot p \cdot n_\text{epochs} \\ &= \text{Cost of Traditional DeepONet.} \end{align*}
The inequality above indicates that D2NO has a training cost similar to that of the traditional DeepONet. Furthermore, this inequality is strict for certain examples presented later in this paper. For instance, when comparing the proposed D2NO (with two clients having a number of sensors $m_1=10$ and $m_2=100$) to the traditional DeepONet, which uses $m=100$ sensors, they both perform an equal number of stochastic gradient updates and use the same amount of training data samples. However, our proposed framework is more efficient because over half of these stochastic gradient updates reduce the input function size by 10x, resulting in a decrease in the computational cost of training.

We conclude this section with the following remark. Although Algorithm \ref{algo_sharing_trunk} uses multiple client networks, which increases the total number of trainable parameters, it is important to note that, as demonstrated in this section, the total number of gradient updates does not increase.

\section{Theoretical Analysis}
\label{sec_theory}
In this section, we demonstrate that the proposed D2NO satisfies a universal approximation theorem. We begin by stating the following lemma~\cite{chen1995universal}, which establishes a universal approximation theorem for functions in a compact subspace with compact support.
\begin{lemma}[\cite{chen1995universal}, Theorem 3]
        Suppose $H$ is compact in $\mathcal{R}^d$, $U\subset C(H)$ is also compact. Let $f\in U$ and for any $\epsilon>0$, there exists an integer $K>0$ independent of $f$,  continuous linear functionals $c_k$ on $U$, networks $b_k: \mathbb{R}^{d} \rightarrow \mathbb{R}$ such that
    \begin{align*}
        \left|f(y) - \sum_{k = 1}^Kc_k(f)b_k \right|<\epsilon,
\end{align*}
for all $y\in H$ and $f\in U$. Specifically, the networks have the following structure, 
\begin{align}
    b_k = g(w_k\cdot y+p_k),
    \label{equation_network_function}
\end{align}
where $w_k\in\mathbb{R}^d$, $p_k\in\mathbb{R}$ and $g$ is a Tauber-Wiener function \cite{chen1995universal}.
\label{chen_theorem3}
\end{lemma}
The next lemma extends the above approximation to functionals defined on a compact subspace of continuous functions.
\begin{lemma}[\cite{chen1995universal}, Theorem 4]
    Suppose that $Y$ is a Banach space, $K\subset Y$ is compact, and $U\subset C(K)$ is also compact.
    Let $f$ be a continuous functional on $U$. For any $\epsilon>0$, there exist weights $c_k$, $\{y_i\}_{i = 1}^I\subset K$ and networks $b_k:\mathbb{R}^{I}\rightarrow \mathbb{R}$ such that,
    \begin{align*}
        \left|f(u) - \sum_{k = 1}^K c_kb_k(\hat{u} ) \right| <\epsilon,
    \end{align*}
    for all $u\in V$, here $\hat{u} = [u(y_1), ..., u(y_I)]^\top$. Specifically, $b_k$ has a structure:
    \begin{align}
        b_k(\hat{u}) = g(w_k\cdot\hat{u}+p_k),
        \label{equation_network_functional}
    \end{align}
    where $w_k\in\mathbb{R}^I$, $p_k\in\mathbb{R}$ and $g$ is a Tauber-Wiener function.
    \label{chen_theorem4}
\end{lemma}
Without loss of generality, we consider two clients ($C=2$) in the remainder of this paper. Let $G: V \rightarrow U$ be a continuous operator. We assume that $G_1: V_1\rightarrow U$ and $G_2: V_2\rightarrow U$ are the restrictions of $G$ on $V_1$ and $V_2$ respectively. Specifically, $G(u) = G_1(u)$ for any $u\in V_1$ and $G(u) = G_2(u)$ for any $u\in V_2$. We will construct the universal approximation for both $G_1$ and $G_2$ using a shared output function basis. Let us make the following assumptions.
\begin{assumption}
\label{assumption_assumption1}
We make the following assumptions on the input/output function spaces and the operator.
    \begin{enumerate}
        \item $V_1\subset C(K_1)$, $V_2\subset C(K_2)$ are compact and are defined on compact domains $K_1\subset K_1'$ and $K_2\subset K_2'$, where $K_1'$ and $K_2'$ are Banach spaces. Additionally, $V = V_1\cup V_2$.
        \item Let $G$, $G_1$ and $G_2$ are continuous, and $G(u) = G_1(u)$ for any $u\in V_1$ and $G(u) = G_2(u)$ for any $u\in V_2$.
        \item $U = G_1(V_1)\cup G_2(V_2)$ \textcolor{black}{has a compact domain} $K_x\subset \mathbb{R}^{d}$.
    \end{enumerate}
\end{assumption}
The following theorem establishes the universal approximation of continuous nonlinear operators.
\begin{theorem}
    Let $G: V\rightarrow U$, $G_1: V_1\rightarrow U$ and $G_2: V_2\rightarrow U$ satisfy the Assumption \ref{assumption_assumption1}. For any $\varepsilon>0$,  there exist weights $c_{i, k}^1, c_{i, k}^2\in \mathbb{R}$, $[y_1, ..., y_{N_1}]^\top \subset K_1$, $[z_1, ..., z_{N_2}]^\top \subset K_2$, networks $b_{i, k}^1: \mathbb{R}^{N_1}\rightarrow\mathbb{R}$, $b_{i, k}^2: \mathbb{R}^{N_2}\rightarrow \mathbb{R}$, specified in Eqaution \ref{equation_network_functional}, and $\tau_k: K_x\rightarrow \mathbb{R}$ specified in \ref{equation_network_function} such that,
    \begin{align*}
        & |G_1(u)(x) - \sum_{k = 1}^{K} \sum_{i = 1}^{I_1}c_{i, k}^1b_{i, k}^1(\hat{u}_1) \tau_k(x)|<\varepsilon, \\
        & |G_2(u)(x) - \sum_{k = 1}^{K} \sum_{i = 1}^{I_2}c_{i, k}^2b_{i, k}^2(\hat{u}_2) \tau_k(x)|<\varepsilon, 
    \end{align*}
    for any $u_1\in U_1$ and $u_2\in U_2$, here $\hat{u}_1 = [u(y_1), ...u(y_{N_1})]^\top$ and $\hat{u}_1 = [u(z_1), ...u(z_{N_2})]^\top$.
\end{theorem}

\begin{proof}
Without loss of generality, we will prove the approximation for the first clients.
For any $\frac{\varepsilon}{2}>0$, by Lemma \ref{chen_theorem3},
there exist $K$ networks $\tau_k(x)$, functional $a_k(\cdot)$ on $V$ such that,
\begin{align}
    |G(u)(x) - \sum_{k = 1}^{K}a_k\left(G(u)\right)\tau_k(x)|<\frac {\varepsilon}{2},
    \label{equation_eqn1}
\end{align}
for any $u\in V$ and $x\in K_x$.
Since $G_1(u) = G(u)$ for $u\in V_1$, it follows that $|G_1(u)(x) - \sum_{k = 1}^{K}a_k\left(G_1(u)\right)\tau_k(x)|<\frac {\varepsilon}{2}$.
Denote $L = \sum_{k = 1}^{K} \sup\limits_{x\in K_x}|\tau_k(x)|$, 
it follows by Lemma \ref{chen_theorem4} that, for $\frac{\varepsilon}{2L}$ and any $k = 1, ..., K$, there exist weights $c^1_{i, k}$, $[y_1, ..., y_{N_1}]\subset K_1$ and networks $b^1_{i, k}: \mathbb{R}^{N_1}\rightarrow \mathbb{R}$ such that,
\begin{align}
    |a_k\left(G(u)\right) - \sum_{i = 1}^{I_1}c^1_{i, k}b^1_{i, k}(\hat{u}_1)|<\frac{\varepsilon}{2L}, \label{equation_eqn2}
\end{align}
where $\hat{u}_1 = [u(y_1), ..., u(y_{N_1})]^\top$. Let $I_1$ large enough and set $c_{i, k}^1 = 0$ if necessary,
substitute Equation \ref{equation_eqn2} into Equation \ref{equation_eqn1}, it follows that,
\begin{align*}
    |G_1(u)(x) - \sum_{k = 1}^{K} \sum_{i = 1}^{I_1}c_{i, k}^1b_{i, k}^1(\hat{u}_1) \tau_k(x)|<\varepsilon,
\end{align*}
for any $u\in V_1$ and $x\in K_x$. The approximation for $G_2$ is similar.
\end{proof}
% ==============================
% Section: Numerical Experiments
% ==============================
\section{Numerical Experiments}
\label{sec_numerical}
In this section, we use four numerical experiments to demonstrate the effectiveness of D2NO in approximating nonlinear operators with heterogeneous input function spaces.

\subsection{Viscous Burgers' equation}
We begin by studying the viscous Burgers' equation. Our main objective is to understand and approximate the operator that maps the initial condition of the PDE to the solution at the terminal time. Specifically, the viscous Burgers' equation is defined as follows:
\begin{align}
&\frac{\partial u_s}{\partial t} + \frac{1}{2}\frac{\partial (u^2_s)}{\partial x} = \alpha \frac{\partial^2 u_s}{\partial x^2},\hspace{0.2em} x\in[0, 2\pi], \hspace{0.2em} t\in[0, T],
\label{eqn_vburgers}
\end{align}
with the initial condition $u_s(x, 0) = u^0_s(x)$ and the boundary conditions $u_s(0, t) = u_s(2\pi, t)$. Here, $u^0_s(x)$ is the initial condition that depends on the parameter $s$, and $\alpha$ is the viscosity of the equation. Our goal is to learn the mapping from $u^0_s$ to $u_s(t)$.

\subsubsection{Heterogeneous input functions with double sharp peaks}
Firstly, we focus our attention on the challenging scenario described in Figure~\ref{fig_different_domain}, where input functions display heterogeneous properties. Specifically, these input functions exhibit distinct sharp peaks, which are concentrated in two separate regions (depicted by the black and magenta curves in Figure~\ref{fig_different_domain}).

\begin{figure}
    \centering
    \includegraphics[scale = 0.4
]{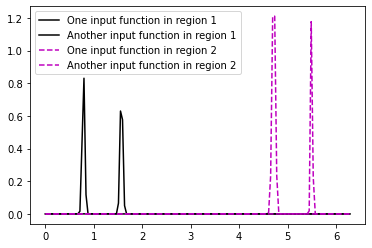}
\includegraphics[scale = 0.4]{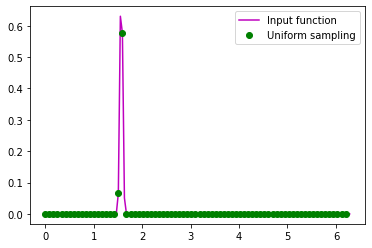}
\includegraphics[scale = 0.4]{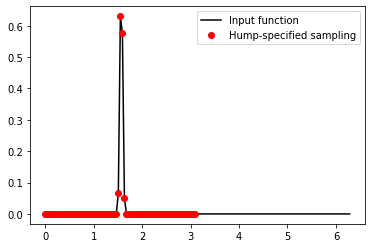}
    \caption{Left: the input functions exhibit peaks, with these peaks being concentrated in two distinct regions (black and magenta). Such a dataset poses significant challenges to the conventional operator learning framework, which mandates uniform sensor placement for all input functions. This is due to the necessity for dense sampling to effectively capture the sparsely distributed peaks. However, this approach results in a substantial increase in problem dimensionality, leading to the presence of numerous redundant entries in the discretized vectors, please check the green dots in the middle image. }
    \label{fig_different_domain}
\end{figure}
In the traditional operator learning or DeepONet framework, uniform sensor placement is required for all input functions. Thus, only a limited number of sensors are used to sample the peak of one input function, and most of these sensors provide the same function values. This uniformity of data significantly affects the quality of training.

By using Algorithm~\ref{algo_sharing_trunk}, we can use two clients ($C=2$) and two dedicated branch networks to capture the input functions individually in the clustered regions. As a result, we can deploy separate sensor sets for these two datasets. Figure \ref{fig_different_domain} illustrates how a function from the left region is sampled using dedicated sensors tailored to that specific region. In contrast, Figure \ref{fig_different_domain} shows the same function sampled using uniform sensors for comparison.

To evaluate the performance of D2NO (see Algorithm~\ref{algo_sharing_trunk}), we trained a total of 50 independent D2NO models and calculated the average $L_2$-relative error. We compared D2NO with the traditional DeepONet model. Specifically, we trained 50 traditional DeepONet models using the same network structure and hyperparameters as D2NO. It is important to note that for the traditional DeepONet, we used uniform sensors for all input functions. Table \ref{table_fed_error} presents a summary of the results, which demonstrate that the proposed D2NO algorithm significantly reduces relative errors while maintaining a comparable computational cost.
\begin{table}[H]
\centering
\begin{tabular}{||c |c| c ||} 
\hline
& DeepONet & D2NO \\
\hline
$L_2$-relative error & $18.64\%$ & $8.10\%$ \\
\hline
\end{tabular}
\caption{The average $L_2$-relative error computed from 50 independent training runs for both the DeepONet and D2NO described in Algorithm~\ref{algo_sharing_trunk}. As a reference, we predict by the average of the training labels, the relative error is $97.59\%$.
}
\label{table_fed_error}
% data from federated_vburgers_data v18
% 18.64 fed_DeepONet_classocal_vburgers v13
% 8.10 fed_DON_share_trunk v16
\end{table}

\subsubsection{Heterogeneous input functions with different properties}
In this experiment, we explore the scenario depicted in Figure \ref{fig_different_geom}. This illustration reflects many other situations where using different numbers of sensors for different input functions can result in improved solutions. Previous attempts have focused on expanding the capabilities of neural operators to discretize different input functions in different ways. However, these efforts have been limited by the need for an equal number of sensors. Our new algorithm overcomes this limitation by allowing the use of varying quantities of sensors for function discretization.

In our example, we present two different types of input functions: (a) functions with humps, represented by the black curve in Figure \ref{fig_different_geom}, and (b) functions with relatively smooth profiles, represented by the magenta curve in the same figure. Typically, the conventional approach involves using the same number of sensors for all functions. However, due to the presence of humps in the first function category, a larger number of sensors is required for both types of functions. The numerical examples demonstrate that DeepONet-style networks can effectively handle challenges presented by smooth input functions, even with a limited number of sensors. However, for functions with peaks, a higher number of sensors is required, which increases the computational cost when dealing with smoother functions.

Since the dedicated branch networks are not shared among clients, our algorithm enables the creation of separate branch networks for these two function classes. As a result, the sensor count and sensor placements can vary depending on the specific function being considered. An illustration of this is in Figure \ref{fig_different_geom}.

\begin{figure}[H]
    \centering
\includegraphics[scale = 0.34]{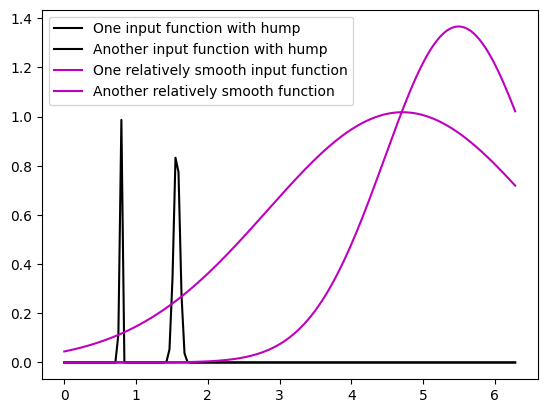}
\includegraphics[scale = 0.34]{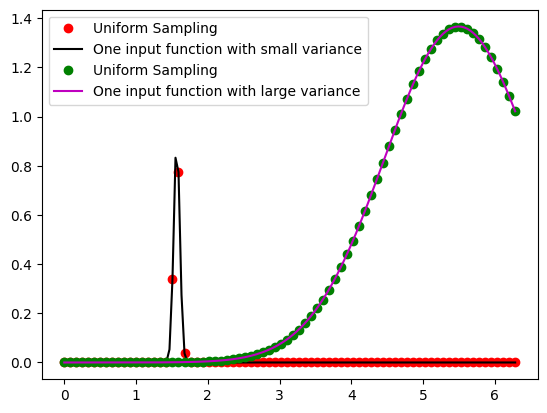}
\includegraphics[scale = 0.34]{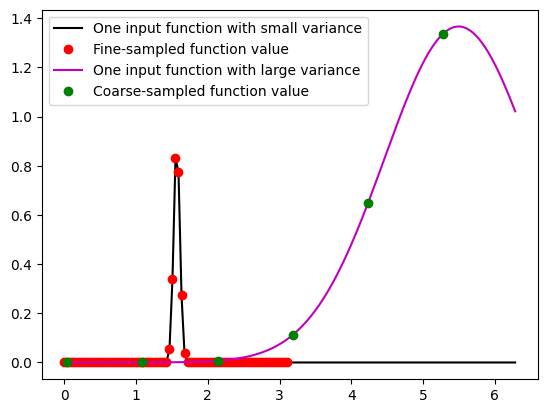}
    \caption{Left: illustrations of functions exhibiting varied degrees of smoothness. The black curves represent two prototypes from a category of input functions characterized by sharp peaks, while the magenta curves display two instances of smooth functions.
Right: presentation of different sensor counts for discretizing functions depicted in the left visual. The functions with humps are finely discretized with $75$ sensors, whereas the smooth magenta function is captured using only $6$ sensors (indicated by the green dots). Middle: uniform sampling.}
    \label{fig_different_geom}
\end{figure}

We use two clients, each with its own dedicated branch network. For functions that have peak characteristics, we use a larger two-layer branch network with dimensions $75\times 100\rightarrow 100\times 1$. On the other hand, for smoother functions, we employ a smaller two-layer branch network with dimensions $6\times 50\rightarrow 50\times 1$. Thus, the proposed D2NO requires the use of 75 sensors and 6 sensors, respectively.

Note that our proposed distributed algorithm (D2NO) adjusts the sensor count for different input function spaces, unlike the traditional DeepONet approach which uses a uniform configuration of 75 sensors for all input functions. Table \ref{table_fed_geom} presents the aggregated results after training 50 independent models using both the traditional DeepONet and the proposed strategy.
\begin{table}[H]
\centering
\begin{tabular}{||c | c| c ||} 
\hline
& DeepONet & D2NO\\
\hline
$L_2$-relative error  &$7.53\%$ & $5.34\%$ \\
\hline
\# Unshared parameters & $77K$ & $41.5K$ \\
\hline
\# Shared parameters  & $51.8K$ & $51.8K$ \\
\hline
\end{tabular}
\caption{The average L2-relative error of 50 independent trainings for DeepONet and D2NO (see Algorithm \ref{algo_sharing_trunk}). Two dedicated branch networks with distinct structures are used to capture hump input functions and smooth input functions, respectively. Note that this approach helps to reduce the number of trainable parameters. As a reference, if we were to predict the solution by the average of the training samples, the relative error would be $810.61\%$.}
\label{table_fed_geom}
% data from federated_vburgers_data v22
% 7.53 fed_DON_classocal_vburgers v19
% 5.34 fed_DON_share_trunk UA account v2, 810.61 is also in this kernel.
\end{table}

% ==============================
% Sub-section: The nonlinear pendulum
% ==============================
\subsection{The Nonlinear Pendulum System} \label{subsec:pendulum-example}
In this experiment, we use the proposed distributed strategy D2NO to approximate the solution operator of the nonlinear pendulum system with dynamics:
\begin{equation}
\begin{aligned}   \label{eq:nonlinear-pendulum}
\frac{ds_1(t)}{dt} &=  s_2(t), \\
\frac{ds_2(t)}{dt} &= -k \sin(s_1(t)) + u(t).
\end{aligned}
\end{equation}
Here, $s(t) = (s_1(t), s_2(t))^\top$ is the state defined over the time domain $t \in [0,1]$ and $u(t)$ is the external force. We aim to approximate the solution operator~$G: u(t) \mapsto s(t)$ for a fixed initial condition $(s_1(0), s_2(0)) = (0.0, 0.0)$, $k=1.0$, and an external force~$u(t)$ sampled from the mean-zero Gaussian Random Field (GRF):
\begin{align} \label{eq:GRF}
    u \sim \mathcal{G}(0, \kappa_{\ell}(x_1, x_2)),
\end{align}
where the covariance kernel $\kappa_{\ell}(x_1,x_2) = \exp\left(-\|x_1 - x_2 \|^2 / 2 \ell^2\right)$ is the radial-basis function (RBF) with length-scale parameter~$\ell > 0$, which we will use to generate heterogeneous input function spaces, as described next. 

\subsubsection{Heterogeneous input functions with different frequencies}
We first consider two input function spaces with different frequencies. To this end, we let the first input space~$V_1$ be a GRF with length scale $\ell = 0.1$ and the second input space~$V_2$ be a GRF with length scale $\ell = 1.0$. Figure~\ref{fig:pendulum-different-freqs-DON} illustrates sample functions from these two input spaces. The traditional DeepONet framework requires placing uniformly a collection of $m$ sensors to discretize the input function~$u$. In general, as depicted in Figure~\ref{fig:pendulum-different-freqs-DON}, a large number of sensors (e.g., $m=100$) is required to capture effectively the input function~$u$, which, in turn, increases the training cost. 

\begin{figure}
    \centering
\includegraphics[scale = 0.4
]{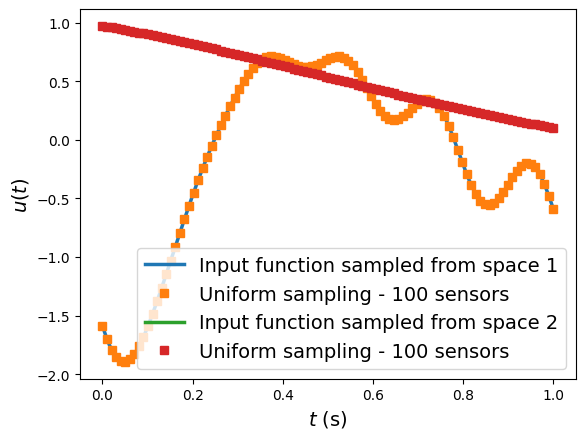}
\includegraphics[scale = 0.4]{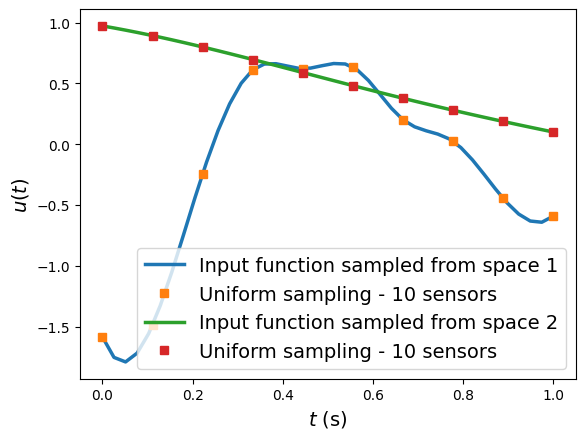}
    \caption{\textit{(Left):} Input functions discretized with uniform sampling and $m=100$ sensors. \textit{(Right):} Input functions discretized with uniform sampling and $m=10$ sensors. The input functions exhibit different frequencies. Such a dataset poses significant challenges to the conventional operator learning framework, which mandates uniform sensor placement for all input functions. This is due to the necessity for dense sampling to effectively capture the oscillatory behavior. However, this approach results in a substantial increase in problem dimensionality, leading to the presence of numerous redundant entries in the discretized vectors.}
    \label{fig:pendulum-different-freqs-DON}
\end{figure}

We note that for the input function sampled from the space~$V_2$, using $m=100$ sensors is excessive. One may require only $m=10$ sensors (see Figure~\ref{fig:pendulum-different-freqs-DON}) to effectively capture input functions sampled from~$V_2$. However, as illustrated in Figure~\ref{fig:pendulum-different-freqs-DON}, using less number of sensors will deteriorate our ability to describe the oscillations from the input function sampled from the input space~$V_1$. The proposed distributed framework D2NO can effectively tackle the above problem by using two clients $(C=2)$, each with dedicated branch networks, as depicted in Figure~\ref{fig:pendulum-different-freqs-Federated}. More specifically, the first client uses a branch network with $m=100$ sensors to capture and discretize input functions sampled from the input space~$V_1$. The second client uses a branch network with $m=10$ sensors to capture and discretize input functions sampled from~$V_2$. We let the two clients share and synchronize the trunk network. 

\begin{figure}[H]
    \centering
\includegraphics[scale = 0.5
]{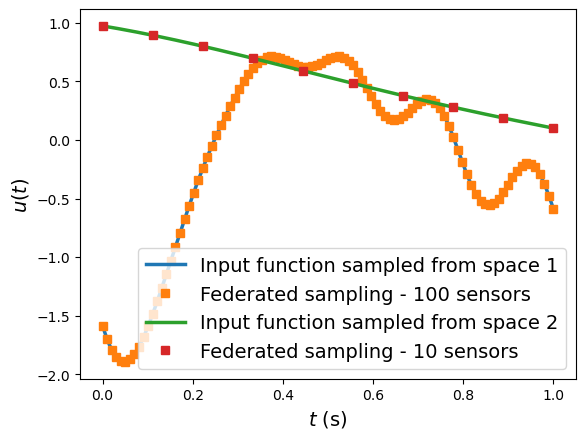}
    \caption{Input functions discretized with federated sampling. A client with dedicated branch networks uses $m=100$ sensors to discretize the input function sampled from a GRF with length-scale $\ell = 0.1$. Another client with a dedicated branch network uses $m=10$ sensors to discretize the input function sampled from a GRF with length-scale $\ell = 1.0$.}
    \label{fig:pendulum-different-freqs-Federated}
\end{figure}

\textit{Training and Testing.} To maintain a fair comparison between the traditional DeepONet and D2NO, we employ the same neural network architectures except for the dedicated branch Networks. We sample $N_{V_1} = N_{V_2} = 1,000$ input trajectories from each input space to train the traditional DeepONet and D2NO. We train the traditional DeepONet using stochastic gradient-based optimization for 1000 epochs. Finally, we train each model within the proposed D2NO approach using 10 synchronization rounds after 100 local stochastic gradient updates. To test the proposed framework, we sample $100$ test trajectories from $V_1$ and $100$ test trajectories from $V_2$.

\textit{Performance of D2NO.} Table~\ref{table:pendulum} shows the performance (i.e., in terms of the $L_2-$ relative error) of the proposed D2NO and the traditional DeepONet (with $m=100$ and $m=10$ sensors) for the test trajectories sampled from the input spaces $V_1$ and $V_2$. The results clearly illustrate that the proposed D2NO framework is competitive with the traditional DeepONet with $m=100$ sensors and vastly outperforms the more efficient DeepONet with $m=10$ sensors.  

\begin{table}[H]
\centering
\begin{tabular}{||c | c | c| c ||} 
\hline
& DeepONet ($m=100$) & DeepONet ($m=10$) & D2NO \\
\hline
$L_2-$relative error on $U_1$ & $3.27\%$ & $16.64\%$ & $5.21\%$ \\
\hline
$L_2-$relative error on $U_2$ & $1.18\%$ & $3.21\%$ & $3.45\%$ \\
\hline
\end{tabular}
\caption{The average $L_2-$relative error for $100$ test trajectories sampled from the input space~$V_1$ (GRF with length-scale $\ell = 0.1$) and for $100$ test trajectories sampled from the input space~$V_2$ (GRF with length-scale $\ell = 1.0$), using the traditional DeepONet ($m=100$ and $m=10$ sensors) and the distributed framework D2NO detailed in Algorithm~\ref{algo_sharing_trunk}.
}
\label{table:pendulum}
\end{table}

\textit{Using Distinct Number of Training Samples}. Here, we test the proposed D2NO framework on a scenario where the number of training trajectories sampled from $V_1$ is $N_{V_1} = 200$ and the number of training trajectories sampled from $V_2$ is $N_{V_2} = 1,800$. \textcolor{black}{Table~\ref{table:pendulum-different-No-train}} describes the performance of the proposed D2NO and the traditional DeepONet. The results clearly illustrate the advantages of the proposed D2NO framework. 

\begin{table}[H]
\centering
\begin{tabular}{||c | c | c| c ||} 
\hline
& DeepONet ($m=100$) & DeepONet ($m=10$) & D2NO \\
\hline
$L_2-$relative error on $U_1$ & $6.34\%$ & $21.23\%$ & $7.08\%$ \\
\hline
$L_2-$relative error on $U_2$ & $0.78\%$ & $2.56\%$ & $3.01\%$ \\
\hline
\end{tabular}
\caption{The average $L_2-$relative error for $100$ test trajectories sampled from the input space~$V_1$ (GRF with length-scale $\ell = 0.1$) and for $100$ test trajectories sampled from the input space~$V_2$ (GRF with length-scale $\ell = 1.0$), using the traditional DeepONet ($m=100$ and $m=10$ sensors) and the distributed framework D2NO detailed in Algorithm~\ref{algo_sharing_trunk}. The frameworks were trained using $N_{V_1} = 200$ input trajectories sampled from $V_1$ and $N_{V_2} = 1,800$ input trajectories sampled from $V_2$.
}
\label{table:pendulum-different-No-train}
\end{table}

\section{Conclusion} \label{sec: conclusion}
In this paper, we have developed a dedicated distributed training strategy for deep neural operators (D2NO) that can efficiently handle heterogeneous input function spaces. Our strategy involves partitioning the input function space based on the properties and regularity of the input functions. Each subset of the input function space is then processed using a dedicated branch network, which is trained locally. Additionally, a consensus trunk network is trained in a centralized manner. This trunk network acts as a global basis for the local coefficients obtained from the dedicated branch networks. We have demonstrated that D2NO serves as a universal approximator of continuous nonlinear operators. Finally, we have demonstrated that D2NO not only addresses the challenges posed by input functions with heterogeneous properties but also offers improved flexibility and computational efficiency using four numerical examples, each with heterogeneous input function spaces.

\section{Acknowledgement}
Z. Zhang was supported in part by AFOSR MURI FA9550-21-1-0084. H. Schaeffer was supported in part by AFOSR MURI FA9550-21-1-0084, NSF DMS-2331033.
G. Lin is supported by the National Science Foundation (DMS-2053746, DMS-2134209, ECCS-2328241, and OAC-2311848), and U.S. Department of Energy (DOE) Office of Science Advanced Scientific Computing Research program DE-SC0021142 and DE-SC0023161.
L. Lu was supported by the U.S. Department of Energy [DE-SC0022953].

\bibliographystyle{abbrv}
\bibliography{references}
\end{document}

%% file: deepo.tex
\begin{tikzpicture}[scale = 1]
 \fill [green!10] (2, 4.5) rectangle (7, 7.6);
  \fill [blue!10] (2, 1.4) rectangle (7, 4.4);
    \fill [red!20] (-3.5, 2) rectangle (1, 7);

    \draw[ultra thick] [decorate,
    decoration = {calligraphic brace, mirror}] (-3.5, 1.8) --  (1, 1.8);
\node at (-1.2, 1.4) {\textcolor{red}{Inputs}};

\fill [magenta!10] (9, 4) rectangle (12, 5);
    \draw[ultra thick] [decorate,
    decoration = {calligraphic brace, mirror}] (9, 4) --  (12, 4);
\node at (10.6, 3.5) {\textcolor{magenta}{Outputs}};

\node[draw, text width=4cm] at (-1.2, 6) {$u = [u(y_1), ..., u(y_N)]^\intercal$};

 \draw [-latex ](1,6) -- (2, 7);
 \node[draw, text width = 2.5cm] at (3.5, 7) {branch net $1$};
 \draw [-latex ](5, 7) -- (6, 7);
 \node[draw, text width = 0.5cm] at (6.5, 7) {$b_1$};
 
 \draw [-latex ](1,6) -- (2, 5);
  \node[draw, text width = 2.5cm] at (3.5, 5) {branch net $K$};
  \draw [-latex ](5, 5) -- (6, 5);
  \node[draw, text width = 0.5cm] at (6.5, 5) {$b_K$};
  
  \draw [-latex ](1,6) -- (2, 6);
 \node[draw, text width=2.5cm] at (3.5, 6) {branch net $j$};
 \draw [-latex ](5, 6) -- (6, 6);
\node[draw, text width = 0.5cm] at (6.5, 6) {$b_j$};

 \node[draw, text width = 2cm] at (-1, 3) {$x\in\mathbb{R}^d$};
  \draw [-latex ](1, 3) -- (2, 3);
 \node[draw, text width = 2.5cm] at (3.5, 3) {trunk net};
 
 \draw [-latex ](5, 3) -- (6, 4);
\node[draw, text width = 0.5cm] at (6.5, 4) {$t_1$};

  \draw [-latex ](5, 3) -- (6, 3);
\node[draw, text width = 0.5cm] at (6.5, 3) {$t_j$};

   \draw [-latex ](5, 3) -- (6, 2);
 \node[draw, text width = 0.5cm] at (6.5, 2) {$t_K$};

\draw [-latex ](7, 6) -- (9, 4.5);
  \draw [-latex ](7, 3) -- (9, 4.5);

  \node[draw, text width = 2.5cm] at (10.5, 4.5) {$\bigotimes\rightarrow G(u)(x)$};

\end{tikzpicture}

%% file: federated_demon.tex
\begin{tikzpicture}[scale = 1]

\node[draw, text width = 5.5cm, fill = blue!10] at (-1.2, 7) {$D^1 = \{[u_i^1(x_1^1), ..., u_i^1(x_{N_1}^1)]^\intercal \}_{i = 1}^{d_1}$};

\node[draw, text width = 5.5cm, fill = blue!10] at (-1.2, 5) {$D^{l} = \{[u_i^{}(x_1^{l}), ..., u_i^{l}(x_{N_{l}}^{l})]^\intercal \}_{i = 1}^{d_{l}}$};

 \draw [blue!15, ultra thick, -> ](2, 5) -- (3, 5);

 \draw [blue!15, ultra thick, -> ](2, 7) -- (3, 7);

 \node[draw, text width = 4.5cm, fill = green] at (5.5, 7) {Unshared networks for $D^1$: $[p^1_1(\cdot),..., p^1_k(\cdot),...]^{\intercal}$ for dataset 1 input functions};

 \node[draw, text width = 4.5cm, fill = green] at (5.5, 5) {Unshared networks for $D^{l}$: $[p^{l}_1(\cdot), ..., p^{l}_k(\cdot), ...]^{\intercal}$ for dataset $l$ input functions};

\draw [green, ultra thick, -> ](8, 7) -- (9.5, 7);
\draw [green, ultra thick, -> ](8, 5) -- (9.5, 7);

\node[draw, text width = 3cm,  fill = brown] at (11.3, 7) {Output: $\sum_{k} p^1_k(u)b_k(x)$, $u\in D^1$ };
\node[draw, text width = 3cm,  fill = brown] at (11.3, 5) {Output: $\sum_{k} p^l_k(u)b_k(x)$, $u\in D^l$ };

\draw [green, ultra thick, -> ](8, 7) -- (9.5, 3);
\draw [green, ultra thick, -> ](8, 5) -- (9.5, 3);

 \node[draw, text width = 5.5cm, fill = blue!10] at (-1.2, 3) {$x\in\mathbb{R}^d$, output independent variable, shared by all};
 \draw [blue!15, ultra thick, ->  ](2, 3) -- (3, 3);
 \node[draw, text width = 4.5cm,  fill = magenta!30] at (5.5, 3) {Shared network $b_k(\cdot)$};

\draw [magenta!30, ultra thick, -> ](8, 3) -- (9.5, 7);
\draw [magenta!30, ultra thick, -> ](8, 3) -- (9.5, 5);
\draw [magenta!30, ultra thick, ->](8, 3) -- (9.5, 3);

 \node[draw, text width = 3cm, fill = brown] at (11.3, 3) {Output: $\sum_{k}p_k^c(u^c_i)b_k(x)$ for all $c$ and $i$};

 \draw [brown!50, ultra thick, -> ](13, 7) -- (14, 7);
\draw [brown!50, ultra thick, -> ](13, 5) -- (14, 5);
\draw [brown!50, ultra thick, ->](13, 3) -- (14, 3);

\node[draw, text width = 3cm,  fill = yellow] at (15.7, 7) {Loss: $L^1(\alpha^1|\beta)$, optimize over $\alpha^1$  };
\node[draw, text width = 3cm,  fill = yellow] at (15.7, 5) {Loss: $L^l(\alpha^l|\beta)$, optimize over $\alpha^l$ };
 \node[draw, text width = 3cm, fill = yellow] at (15.7, 3) {Global loss: $L(\beta|\alpha)$, optimize over $\beta$};

\end{tikzpicture}

%% file: centralized_demon.tex
\begin{tikzpicture}[scale = 1]

\node[draw, text width = 4.5cm, fill = blue!10] at (-1.2, 5) {Input functions of all datasets $\bigcup\limits_{i}D^i$};

 \draw [blue!15, ultra thick, -> ](1.3, 5) -- (3, 5);

 \node[draw, text width = 5.5cm, fill = green] at (6, 5) {Branch net $p_k(u)$, dependent on the input functions};

\node[draw, text width = 6cm,  fill = yellow] at (14., 4.) {Network output: $\sum\limits_{k}p_k(u)b_k(x)$. Loss function: $\sum\limits_{u, x}\|\sum\limits_{k}p_k(u)b_k(x) - G(u)(x)\|_2^2$, the summation is over \textcolor{blue}{ALL $u$ and $x$} };

\draw [green, ultra thick, -> ](9, 5) -- (10.5, 4);

 \node[draw, text width = 4.5cm, fill = blue!10] at (-1.2, 3) {$x\in\mathbb{R}^d$};
 \draw [blue!15, ultra thick, ->  ](1.3, 3) -- (3, 3);
 \node[draw, text width = 5.5cm,  fill = magenta!30] at (6, 3) {Construction net $b_k(x)$, independent of the input functions};

\draw [magenta!30, ultra thick, ->](9, 3) -- (10.5, 4);

\end{tikzpicture}

%% file: main.bbl
\begin{thebibliography}{10}

\bibitem{bhan2023operator}
L.~Bhan, Y.~Shi, and M.~Krstic.
\newblock Operator learning for nonlinear adaptive control.
\newblock In {\em Learning for Dynamics and Control Conference}, pages
  346--357. PMLR, 2023.

\bibitem{cai2021deepm}
S.~Cai, Z.~Wang, L.~Lu, T.~A. Zaki, and G.~E. Karniadakis.
\newblock Deepm\&mnet: Inferring the electroconvection multiphysics fields
  based on operator approximation by neural networks.
\newblock {\em Journal of Computational Physics}, 436:110296, 2021.

\bibitem{cao2022residual}
L.~Cao, T.~O'Leary-Roseberry, P.~K. Jha, J.~T. Oden, and O.~Ghattas.
\newblock Residual-based error correction for neural operator accelerated
  infinite-dimensional bayesian inverse problems.
\newblock {\em arXiv preprint arXiv:2210.03008}, 2022.

\bibitem{chen1993approximations}
T.~Chen and H.~Chen.
\newblock Approximations of continuous functionals by neural networks with
  application to dynamic systems.
\newblock {\em IEEE Transactions on Neural networks}, 4(6):910--918, 1993.

\bibitem{chen1995universal}
T.~Chen and H.~Chen.
\newblock Universal approximation to nonlinear operators by neural networks
  with arbitrary activation functions and its application to dynamical systems.
\newblock {\em IEEE Transactions on Neural Networks}, 6(4):911--917, 1995.

\bibitem{dean2012large}
J.~Dean, G.~Corrado, R.~Monga, K.~Chen, M.~Devin, M.~Mao, M.~Ranzato,
  A.~Senior, P.~Tucker, K.~Yang, et~al.
\newblock Large scale distributed deep networks.
\newblock {\em Advances in neural information processing systems}, 25, 2012.

\bibitem{deng2022approximation}
B.~Deng, Y.~Shin, L.~Lu, Z.~Zhang, and G.~E. Karniadakis.
\newblock Approximation rates of deeponets for learning operators arising from
  advection--diffusion equations.
\newblock {\em Neural Networks}, 153:411--426, 2022.

\bibitem{deng2021convergence}
W.~Deng, Q.~Zhang, Y.-A. Ma, Z.~Song, and G.~Lin.
\newblock On convergence of federated averaging langevin dynamics.
\newblock {\em arXiv preprint arXiv:2112.05120}, 2021.

\bibitem{di2023neural}
P.~C. Di~Leoni, L.~Lu, C.~Meneveau, G.~E. Karniadakis, and T.~A. Zaki.
\newblock Neural operator prediction of linear instability waves in high-speed
  boundary layers.
\newblock {\em Journal of Computational Physics}, 474:111793, 2023.

\bibitem{howard2022multifidelity}
A.~A. Howard, M.~Perego, G.~E. Karniadakis, and P.~Stinis.
\newblock Multifidelity deep operator networks.
\newblock {\em arXiv preprint arXiv:2204.09157}, 2022.

\bibitem{jiang2023fourier}
Z.~Jiang, M.~Zhu, D.~Li, Q.~Li, Y.~O. Yuan, and L.~Lu.
\newblock Fourier-mionet: Fourier-enhanced multiple-input neural operators for
  multiphase modeling of geological carbon sequestration.
\newblock {\em arXiv preprint arXiv:2303.04778}, 2023.

\bibitem{jin2022mionet}
P.~Jin, S.~Meng, and L.~Lu.
\newblock Mionet: Learning multiple-input operators via tensor product.
\newblock {\em SIAM Journal on Scientific Computing}, 44(6):A3490--A3514, 2022.

\bibitem{karimireddy2020scaffold}
S.~P. Karimireddy, S.~Kale, M.~Mohri, S.~Reddi, S.~Stich, and A.~T. Suresh.
\newblock Scaffold: Stochastic controlled averaging for federated learning.
\newblock In {\em International Conference on Machine Learning}, pages
  5132--5143. PMLR, 2020.

\bibitem{karniadakis2021physics}
G.~E. Karniadakis, I.~G. Kevrekidis, L.~Lu, P.~Perdikaris, S.~Wang, and
  L.~Yang.
\newblock Physics-informed machine learning.
\newblock {\em Nature Reviews Physics}, 3(6):422--440, 2021.

\bibitem{khaled2019first}
A.~Khaled, K.~Mishchenko, and P.~Richt{\'a}rik.
\newblock First analysis of local gd on heterogeneous data.
\newblock {\em arXiv preprint arXiv:1909.04715}, 2019.

\bibitem{lanthaler2023curse}
S.~Lanthaler and A.~M. Stuart.
\newblock The curse of dimensionality in operator learning.
\newblock {\em arXiv preprint arXiv:2306.15924}, 2023.

\bibitem{lsts20}
T.~Li, A.~K. Sahu, A.~Talwalkar, and V.~Smith.
\newblock Federated {L}earning: Challenges, {M}ethods, and {F}uture
  {D}irections.
\newblock {\em IEEE Signal Processing Magazine}, 37(3):50--60, 2020.

\bibitem{li2019convergence}
X.~Li, K.~Huang, W.~Yang, S.~Wang, and Z.~Zhang.
\newblock On the convergence of fedavg on non-iid data.
\newblock {\em arXiv preprint arXiv:1907.02189}, 2019.

\bibitem{li2021fedbn}
X.~Li, M.~Jiang, X.~Zhang, M.~Kamp, and Q.~Dou.
\newblock Fedbn: Federated learning on non-iid features via local batch
  normalization.
\newblock {\em arXiv preprint arXiv:2102.07623}, 2021.

\bibitem{li2020fourier}
Z.~Li, N.~Kovachki, K.~Azizzadenesheli, B.~Liu, K.~Bhattacharya, A.~Stuart, and
  A.~Anandkumar.
\newblock Fourier neural operator for parametric partial differential
  equations.
\newblock {\em arXiv preprint arXiv:2010.08895}, 2020.

\bibitem{lin2021operator}
C.~Lin, Z.~Li, L.~Lu, S.~Cai, M.~Maxey, and G.~E. Karniadakis.
\newblock Operator learning for predicting multiscale bubble growth dynamics.
\newblock {\em The Journal of Chemical Physics}, 154(10), 2021.

\bibitem{lin2021accelerated}
G.~Lin, C.~Moya, and Z.~Zhang.
\newblock Accelerated replica exchange stochastic gradient langevin diffusion
  enhanced bayesian deeponet for solving noisy parametric pdes.
\newblock {\em arXiv preprint arXiv:2111.02484}, 2021.

\bibitem{lin2023b}
G.~Lin, C.~Moya, and Z.~Zhang.
\newblock B-deeponet: An enhanced bayesian deeponet for solving noisy
  parametric pdes using accelerated replica exchange sgld.
\newblock {\em Journal of Computational Physics}, 473:111713, 2023.

\bibitem{lin2023learning}
G.~Lin, C.~Moya, and Z.~Zhang.
\newblock Learning the dynamical response of nonlinear non-autonomous dynamical
  systems with deep operator neural networks.
\newblock {\em Engineering Applications of Artificial Intelligence},
  125:106689, 2023.

\bibitem{liu2023prose}
Y.~Liu, Z.~Zhang, and H.~Schaeffer.
\newblock Prose: Predicting operators and symbolic expressions using multimodal
  transformers.
\newblock {\em arXiv preprint arXiv:2309.16816}, 2023.

\bibitem{lu2021learning}
L.~Lu, P.~Jin, G.~Pang, Z.~Zhang, and G.~E. Karniadakis.
\newblock Learning nonlinear operators via deeponet based on the universal
  approximation theorem of operators.
\newblock {\em Nature Machine Intelligence}, 3(3):218--229, 2021.

\bibitem{lu2022comprehensive}
L.~Lu, X.~Meng, S.~Cai, Z.~Mao, S.~Goswami, Z.~Zhang, and G.~E. Karniadakis.
\newblock A comprehensive and fair comparison of two neural operators (with
  practical extensions) based on fair data.
\newblock {\em Computer Methods in Applied Mechanics and Engineering},
  393:114778, 2022.

\bibitem{lu2022multifidelity}
L.~Lu, R.~Pestourie, S.~G. Johnson, and G.~Romano.
\newblock Multifidelity deep neural operators for efficient learning of partial
  differential equations with application to fast inverse design of nanoscale
  heat transport.
\newblock {\em Physical Review Research}, 4(2):023210, 2022.

\bibitem{mao2023ppdonet}
S.~Mao, R.~Dong, L.~Lu, K.~M. Yi, S.~Wang, and P.~Perdikaris.
\newblock Ppdonet: Deep operator networks for fast prediction of steady-state
  solutions in disk--planet systems.
\newblock {\em The Astrophysical Journal Letters}, 950(2):L12, 2023.

\bibitem{mao2021deepm}
Z.~Mao, L.~Lu, O.~Marxen, T.~A. Zaki, and G.~E. Karniadakis.
\newblock Deepm\&mnet for hypersonics: Predicting the coupled flow and
  finite-rate chemistry behind a normal shock using neural-network
  approximation of operators.
\newblock {\em Journal of computational physics}, 447:110698, 2021.

\bibitem{mcmahan2017communication}
B.~McMahan, E.~Moore, D.~Ramage, S.~Hampson, and B.~A. y~Arcas.
\newblock Communication-efficient learning of deep networks from decentralized
  data.
\newblock In {\em Artificial intelligence and statistics}, pages 1273--1282.
  PMLR, 2017.

\bibitem{mcmahan2016federated}
H.~B. McMahan, E.~Moore, D.~Ramage, and B.~A. y~Arcas.
\newblock Federated learning of deep networks using model averaging.
\newblock {\em arXiv preprint arXiv:1602.05629}, 2, 2016.

\bibitem{moya2022fed}
C.~Moya and G.~Lin.
\newblock Fed-deeponet: Stochastic gradient-based federated training of deep
  operator networks.
\newblock {\em Algorithms}, 15(9):325, 2022.

\bibitem{moya2023deeponet}
C.~Moya, S.~Zhang, G.~Lin, and M.~Yue.
\newblock Deeponet-grid-uq: A trustworthy deep operator framework for
  predicting the power grid’s post-fault trajectories.
\newblock {\em Neurocomputing}, 535:166--182, 2023.

\bibitem{pathak2022fourcastnet}
J.~Pathak, S.~Subramanian, P.~Harrington, S.~Raja, A.~Chattopadhyay,
  M.~Mardani, T.~Kurth, D.~Hall, Z.~Li, K.~Azizzadenesheli, et~al.
\newblock Fourcastnet: A global data-driven high-resolution weather model using
  adaptive fourier neural operators.
\newblock {\em arXiv preprint arXiv:2202.11214}, 2022.

\bibitem{psaros2023uncertainty}
A.~F. Psaros, X.~Meng, Z.~Zou, L.~Guo, and G.~E. Karniadakis.
\newblock Uncertainty quantification in scientific machine learning: Methods,
  metrics, and comparisons.
\newblock {\em Journal of Computational Physics}, 477:111902, 2023.

\bibitem{qian2022reduced}
E.~Qian, I.-G. Farcas, and K.~Willcox.
\newblock Reduced operator inference for nonlinear partial differential
  equations.
\newblock {\em SIAM Journal on Scientific Computing}, 44(4):A1934--A1959, 2022.

\bibitem{sahin2024deep}
I.~Sahin, C.~Moya, A.~Mollaali, G.~Lin, and G.~Paniagua.
\newblock Deep operator learning-based surrogate models with uncertainty
  quantification for optimizing internal cooling channel rib profiles.
\newblock {\em International Journal of Heat and Mass Transfer}, 219:124813,
  2024.

\bibitem{sun2023deepgraphonet}
Y.~Sun, C.~Moya, G.~Lin, and M.~Yue.
\newblock Deepgraphonet: A deep graph operator network to learn and zero-shot
  transfer the dynamic response of networked systems.
\newblock {\em IEEE Systems Journal}, 2023.

\bibitem{wang2021learning}
S.~Wang, H.~Wang, and P.~Perdikaris.
\newblock Learning the solution operator of parametric partial differential
  equations with physics-informed deeponets.
\newblock {\em Science advances}, 7(40):eabi8605, 2021.

\bibitem{yang2022scalable}
Y.~Yang, G.~Kissas, and P.~Perdikaris.
\newblock Scalable uncertainty quantification for deep operator networks using
  randomized priors.
\newblock {\em Computer Methods in Applied Mechanics and Engineering},
  399:115399, 2022.

\bibitem{ye2023deeponet}
K.~Ye, J.~Zhao, X.~Liu, C.~Moya, and G.~Lin.
\newblock Deeponet based uncertainty quantification for power system dynamics
  with stochastic loads.
\newblock In {\em 2023 IEEE Power \& Energy Society General Meeting (PESGM)},
  pages 1--6. IEEE, 2023.

\bibitem{zhang2022belnet}
Z.~Zhang, W.~T. Leung, and H.~Schaeffer.
\newblock Belnet: Basis enhanced learning, a mesh-free neural operator.
\newblock {\em arXiv preprint arXiv:2212.07336}, 2022.

\bibitem{zhang2023discretization}
Z.~Zhang, W.~T. Leung, and H.~Schaeffer.
\newblock A discretization-invariant extension and analysis of some deep
  operator networks.
\newblock {\em arXiv preprint arXiv:2307.09738}, 2023.

\bibitem{zhang2023bayesian}
Z.~Zhang, C.~Moya, W.~T. Leung, G.~Lin, and H.~Schaeffer.
\newblock Bayesian deep operator learning for homogenized to fine-scale maps
  for multiscale pde.
\newblock {\em arXiv preprint arXiv:2308.14188}, 2023.

\bibitem{zhu2023fourier}
M.~Zhu, S.~Feng, Y.~Lin, and L.~Lu.
\newblock Fourier-deeponet: Fourier-enhanced deep operator networks for full
  waveform inversion with improved accuracy, generalizability, and robustness.
\newblock {\em arXiv preprint arXiv:2305.17289}, 2023.

\bibitem{zhu2023reliable}
M.~Zhu, H.~Zhang, A.~Jiao, G.~E. Karniadakis, and L.~Lu.
\newblock Reliable extrapolation of deep neural operators informed by physics
  or sparse observations.
\newblock {\em Computer Methods in Applied Mechanics and Engineering},
  412:116064, 2023.

\end{thebibliography}
